\documentclass[12pt, english, a4paper]{article}
\usepackage[latin9]{inputenc}
\usepackage[T1]{fontenc}
\usepackage{graphicx}
\usepackage{amsthm}
\usepackage{amsfonts}
\usepackage{amsmath}
\usepackage{amssymb}
\usepackage{makecell}
\usepackage{cases}
\usepackage[mathscr]{eucal}
\usepackage{fullpage}
\usepackage{times}
\usepackage[usenames]{color}
\usepackage[dvipsnames]{xcolor}
\usepackage{hyperref}

\newcommand{\SSS}{\mathfrak{S}}

\newcommand{\stir}{ \mathcal{Q}}
\newcommand{\seqnum}[1]{\href{https://oeis.org/#1}{\rm \underline{#1}}}

\DeclareMathOperator{\flt}{flat}
\DeclareMathOperator{\maj}{maj}

\DeclareMathOperator{\des}{des}
\DeclareMathOperator{\inv}{inv}
\DeclareMathOperator{\exc}{exc}
\DeclareMathOperator{\run}{run}

\theoremstyle{plain}
\newtheorem{theorem}{Theorem}
\newtheorem{corollary}[theorem]{Corollary}
\newtheorem{lemma}[theorem]{Lemma}
\newtheorem{proposition}[theorem]{Proposition}
\theoremstyle{definition}
\newtheorem{definition}[theorem]{Definition}
\newtheorem{example}[theorem]{Example}
\newtheorem{conjecture}[theorem]{Conjecture}
\theoremstyle{remark}
\newtheorem{remark}[theorem]{Remark}

\usepackage{authblk}
\title{Enumeration of flattened $k$-Stirling permutations with respect to descents}
\author[1]{Umesh Shankar\thanks{\tt{204093001@iitb.ac.in,umeshshankar@outlook.com}}} 
\affil[1]{Department of Mathematics, Indian Institute of Technology, Bombay Mumbai 400076, India} 
\date{July, 2023}
\begin{document}
\maketitle
\begin{abstract}
    A $k$-Stirling permutation of order $n$ is said to be "flattened" if the leading terms of its increasing runs are in ascending order. We show that flattened $k$-Stirling permutations of order $n+1$ are in bijective correspondence with a colored variant of type $B$ set partitions of $[-n,n]$. Using the theory of weighted labelled structures, we give the exponential generating function of their descent enumerating polynomials. We also provide enumerative formulae for the number of flattened $k$-Stirling permutations of order $n$ with small number of descents and the number of flattened $2$-Stirling permutations with maximum number of descents.    
\end{abstract}
\textbf{\small{}Keywords:}{\small{} flattened partitions, $k$-Stirling permutations, colored set partition, type $B$ set partitions, exponential generating function, enumeration}{\let\thefootnote\relax\footnotetext{The author is supported by the National Board for Higher Mathematics, India.}}{\let\thefootnote\relax}

\section{Introduction}
For a positive integer $n$, let $[n]$ be the set $\lbrace 1,2,\dots,n \rbrace$ and $\SSS_n$ be the set of permutations of $[n]$. The study of permutation statistics has been around for more than a century with most experts attributing its inception to the work \cite{Macmajor} of MacMahon who studied the four statistics for a permutation $\pi$, the number of descents $(\des(\pi))$, the number of excedances $(\exc(\pi))$, the major index $(\maj(\pi))$ and the number of inversions $(\inv(\pi))$. As we use descents in this work, we define them.  For a permutation $\pi=\pi_1\pi_2\dots\pi_n$, a descent of $\pi$ is an index $i\in [n-1]$ 
such that $\pi_i>\pi_{i+1}$.

The classical Eulerian polynomial $A_n(t)$ is the polynomial that enumerates descents over $\SSS_n$. It has been extensively studied in \cite{foata-schutzenberger-eulerian,foata-desarm_loday,branden-actions_on_perms_unimodality_descents,foata-strehl-actiontansec,foata-strehl} and is shown to have a great many remarkable properties. The book \cite{petersen-eulerian-nos-book} by Petersen is a good reference on this subject. One of these interesting properties is the Carlitz identity.
\begin{theorem}[Carlitz Identity] For any $n\ge 0$ and $A_n(t)$, the Eulerian polynomial,
    \begin{equation}
        \frac{tA_n(t)}{(1-t)^{n+1}}=\displaystyle \sum_{r\ge 1} r^nt^r.
    \end{equation}
\end{theorem}
The set of Stirling permutations of order $n$ is the set of permutations on the multiset $\lbrace 1,1,2,2,\dots, n,n \rbrace$ that satisfy the additional property that, for all $1\le i\le n$, the values between the occurrences of $i$ are all larger than $i$. Gessel and Stanley, in \cite{gessel-stan-stirling}, proved that the descent enumerating polynomial $Q_n(t)$ over the Stirling permutations satisfies the following analogue of the Carlitz identity.
\begin{theorem}[Gessel, Stanley] For any $n\ge 0$,
    \begin{equation}
        \frac{tQ_n(t)}{(1-t)^{2n+1}}=\displaystyle \sum_{r\ge 0} S(r+n,r)t^r
    \end{equation}
    where $S(a,b)$ is the Stirling numbers of the second kind with $S(a,0)=0$ for $a\ge 0$.
\end{theorem}

They also defined a very natural generalization of Stirling permutations. Define $j^k$ to be $\underbrace{jj\dots j}_{\text{$k$ copies}}$ for $j,k \in \mathbb N$ . We call a permutation of the multiset $\lbrace 1^k,2^k,\dots,n^k\rbrace $ (we have $k$ copies of each letter) a $k$-Stirling permutation of order $n$ if for each $i$, $1 \le i \le n$, all entries between two consecutive occurrences of $i$ are greater than $i$. We shall denote this set of $k$-Stirling permutations of order $n$ by $\stir_n^k$. By this definition, the set $\stir^1_n=\SSS_n$ and $\stir_n^2$ is the set of Stirling permutations. 

Callan in \cite{callan-flatten} defined the notion of flattening a set partition into a permutation by erasing the dividers between its blocks. Since the order depends on the listing of the blocks and internal arrangement of letters of the blocks, the letters are written in increasing order inside each block and the blocks are arranged from left to right in increasing order of minimum element. Therefore, a \textit{flattened partition} of $n$ is a permutation in $\SSS_n$ that is obtained by flattening a set partition of $[n]$.  Mansour in \cite{mansour-flatten} also uses the same notion.

For a permutation $\pi \in \SSS_n$, a \textit{run} is a maximal contiguous increasing or decreasing subword of $\pi$. However, in this work, runs are synonymous with increasing runs. It should be noted that enumeration of various sets of permutations by number of runs is abundant in the literature of this subject as can be seen in \cite{maaltrun,bonaaltrun,canfieldwilfalternatingrun,Andrealtrun}.

While reading from left to right, if the sequence of leading letters of the runs are in increasing order, we get a flattened partition of $n$. Nabawanda, Rakotondrajao, and Bamunoba, in \cite{nabawanda-flatten}, 
 showed that the set of flattened partitions of $n$ are in bijection with set partitions of $[n-1]$. Furthermore, they also proved a recurrence that establishes a generating function for the run-enumerating polynomial over the flattened partitions.

 For $w \in \stir_n^k$, a $k$-Stirling permutation of order $n$, Buck et al., in \cite{buck2023flattened}, define a \textit{run} of $w$ to be a maximal contiguous weakly increasing subword of $w$. Similarly, if the leading terms of the runs of $w$ are in weakly increasing order from left to right, they call such a $k$-Stirling permutation a \textit{flattened $k$-Stirling permutation}. We denote the set of flattened $k$-Stirling permutations of order $n$ by $\flt(\stir_n^k)$. Buck et al., studied the flattened $2$-Stirling permutations and showed that the flattened $2$-Stirling permutations of order $[n+1]$ are in bijection with the type $B$ set partitions of $[-n,n]$.

Due to this connection between the flattened $k$-Stirling permutations and set partitions, for $k=1,2$, it is natural to wonder whether there is a structure on the set partitions that is in bijection with the flattened $k$-Stirling permutations for arbitrary $k$. One of the objectives in this paper is to give a bijection between the flattened $k$-Stirling permutations and a special case of $(c,m)$-colored $B_n$ partitions, introduced by D.G.L. Wang in \cite{wang-colored-set}. This is done in Definition \ref{def: bij} in Section \ref{sec: bijection}.

We will show that flattened $k$-Stirling permutations of order $n$ are in bijective correspondence with a slightly modified variant of the $(k-1,k)$-colored $B_{n-1}$ partitions. 
From this bijection, using the theory of labelled structures (combinatorial species), we establish the following closed form expression for the exponential generating function (egf, henceforth) $$F_k(z):=\displaystyle \sum_{n\ge 0} \vert \flt(\stir_{n+1}^k) \vert \frac{z^n}{n!}.$$

\begin{theorem}\label{thm:egf} For $k\in \mathbb N, k\ge 1$, we have
    \begin{equation}
        F_k(z)=\exp\bigg((k-1)z+\frac{\exp(kz)-1}{k}\bigg).
    \end{equation}
\end{theorem}

As a consequence of our egf, we obtain the following formula for the number of flattened $k$-Stirling permutations of order $n$. This affirms {\cite[Conjecture 5.2]{buck2023flattened}}.
\begin{corollary} [{\cite[Conjecture 5.2]{buck2023flattened}}] \label{thm: series-flat} For integers $k\ge 1$ and $n\ge 0$, we have
    \begin{equation}
         \vert \flt(\stir^k_{n+1}) \vert = e^{-1/k}\displaystyle \sum_{r\ge 0} \frac{(kr+k-1)^n}{r!k^r},  
    \end{equation}
    \end{corollary}

The proofs of Theorem \ref{thm:egf} and Corollary \ref{thm: series-flat} appear in Section \ref{section: proof-corollary}. For $k=2,3,4$, the sequence $\vert \flt(\stir_{n+1}^k)\vert$, and therefore, the number of $(k-1,k)$-colored $B_n$ partitions, matches the OEIS sequences given by \seqnum{A007405}, \seqnum{A355164}, \seqnum{A355167} respectively. In Section \ref{section: proof-corollary}, we show that our bijection also yields the following identity that connects the numbers $\vert \flt(\stir_n^k) \vert$ to the Stirling numbers of the second kind.
\begin{theorem}\label{thm:bij_rec} If $n\ge 0$, then
\begin{equation}
    \vert \flt(\stir_{n+1}^k) \vert = \sum_{i=0}^n \binom{n}{i} (k-1)^{i}\sum_{r=0}^{n-i} k^{n-i-r} S(n-i,r)
\end{equation}
where $S(a,b)$ denotes the Stirling number of the second kind with $S(a,0)=0$ for all $a\ge 0$.
\end{theorem}

In the same section, we also prove a recurrence relation for the numbers $\vert \flt(\stir_n^k) \vert$ that was conjectured in \cite{buck2023flattened}.
\begin{theorem}[{\cite[Conjecture 5.2]{buck2023flattened}}]\label{thm: recurrence}
    For $n\ge 1$, the numbers $\vert \flt(\stir^k_{n+1})\vert$ satisfy the recurrence relation 
\begin{equation}
    \vert \flt(\stir^k_{n+1}) \vert= (k-1)\vert \flt(\stir^k_{n}) \vert+\displaystyle \sum_{r=1}^{n} \binom{n-1}{r-1} k^{r-1} \vert \flt(\stir^k_{n-r+1}) \vert, 
\end{equation} with $\vert \flt(\stir^k_1)\vert=1$.
\end{theorem}

One of our main results requires the following definition.

\begin{definition} For $n\ge 0$, we define the $n^{th}$ $\flt$-Eulerian polynomial of order $k$
    $$A^{\flt}_{n,k}(t):=\displaystyle \sum_{W\in \flt(\stir^k_n)} t^{\des(W)}.$$
\end{definition}

In Section \ref{section: prelim}, we give the relevant background from the theory of weighted species/weighted labelled structures. Using this, in Section \ref{section: descent-egf}, we obtain the following egf. 

\begin{theorem}
    \label{thm: main-egf}
For $k\ge 1$ and $n\ge 0$,
    \begin{equation}
          \displaystyle \sum_{n\ge 0} A^{\flt}_{n+1,k}(t) \frac{z^n}{n!}= (t(e^z-1)+1)^{k-1}\exp\Bigg(z+\displaystyle \sum_{j=1}^{k} \frac{k!}{(k-j)!}H_j(z)t^j\Bigg)
    \end{equation}
where $H_j(z):=\displaystyle \sum_{n=0}^{\infty} S(n-1,j)\frac{z^n}{n!}$ and $S(a,b)$ are the Stirling numbers of the second kind with $S(a,0)=0$ for $a\ge 0$.  
\end{theorem}

As a corollary, by setting $k=1$ in Theorem \ref{thm: main-egf}, we obtain a different proof of Theorem 13 from \cite{nabawanda-flatten}.
\begin{corollary}[{\cite[Theorem 13]{nabawanda-flatten}}] \label{thm:naba}
We have the following egf of the $\flt$-Eulerian polynomials of order $1$.
    \begin{equation}
        \displaystyle \sum_{n\ge 0} A^{\flt}_{n+1,1}(t) \frac{z^n}{n!}= e^{z+t(e^z-z-1)}.
    \end{equation} 
\end{corollary}
As another corollary, by setting $k=2$ in Theorem \ref{thm: main-egf}, we get the following egf which solves an open problem from \cite{buck2023flattened}.
\begin{corollary} \label{thm:stirling} We have the following egf of the $\flt$-Eulerian polynomials of order $2$.
    \begin{equation}
        \displaystyle \sum_{n\ge 0} A^{\flt}_{n+1,2}(t) \frac{z^n}{n!}= (t(e^z-1)+1)\exp\bigg(z+2t(e^z-z-1)+2t^2\bigg(\frac{3 +2 z - 4 e^z + e^{2 z}}{4}\bigg)\bigg).
    \end{equation}
\end{corollary}
 
In \cite{buck2023flattened}, the authors give the following formula for the number of flattened $2$-Stirling permutations with exactly $2$ runs. 
\begin{proposition}[{\cite[Corollary 4.5]{buck2023flattened}}]\label{prop: buck-2-runs}
    The number of flattened $2$-Stirling permutations of order $n$ with exactly two runs is $$|\flt_2(\stir^2_{n})|=3(2^{n-1}-1)-2(n-1).$$
\end{proposition}
In Section \ref{sec: run-refinement}, we generalize this to flattened $k$-Stirling permutations by showing the following result. It is easy to see that Proposition \ref{prop: 2_runs} gives Proposition \ref{prop: buck-2-runs} when $k=2$.
\begin{proposition}\label{prop: 2_runs} For $n\ge 0$, $k\ge 1$,
    $$\vert \flt_2(\stir_n^k) \vert=(2k-1)(2^{n-1}-1)-k(n-1).$$
\end{proposition}
The authors of \cite{buck2023flattened} conjectured the following formula for the number of flattened $2$-Stirling permutations with exactly $3$ runs.
\begin{conjecture}[{\cite[Conjecture $5.1$]{buck2023flattened}}] \label{conj: buck-3-runs}    
 The number of flattened $2$-Stirling permutations of order $n$ with exactly three runs is
 \begin{eqnarray}
      |\flt_3(\stir_n)| &=& 2\sum_{i = 1}^{n-1} \binom{n-1}{i} \left( \sum_{j = 2}^{n-1 - i} \binom{n-1 -i}{j} \right) 
    \nonumber \\
    &+& \sum_{i = 2}^{n-1} \binom{n-1}{i} \left( \sum_{i = 2}^{n-1 - i} \binom{n-1 -i}{j} \right)
    + \sum_{i = 3}^{n-1} (2^{i-1} - 2)\binom{n-1}{i}.
 \end{eqnarray}
\end{conjecture}

In Section \ref{sec: run-refinement}, we prove the following generalization. It is clear that our formula in Proposition \ref{prop:3_runs} gives the expression in Conjecture \ref{conj: buck-3-runs} when $k=2$.

\begin{proposition} \label{prop:3_runs}
    For $n\ge 0$ and $k\ge 1$,
\begin{eqnarray*}
    \vert \flt_3(\stir_n^k) \vert
&=& \frac{(k-1)(k-2)}{2}(3^{n-1}-2^n+1)  +  \frac{k(k-1)}{12}(3^n-6\times 2^{n}+6n+3)\\  &+&  k(k-1)\sum_{i=1}^{n-3}\binom{n-1}{i}\sum_{j=2}^{n-1-i}\binom{n-1-i}{j}  +  \frac{k^2}{2}\sum_{i=2}^{n-3}\binom{n-1}{i}\sum_{j=2}^{n-1-i}\binom{n-1-i}{j}
\end{eqnarray*}
\end{proposition}

In {\cite[Open Problem 2]{buck2023flattened}}, the authors ask for the number of flattened $2$-Stirling permutations that have maximum number of runs. Denote by $\vert \flt_{\max}(\stir_n^2) \vert$ the number of flattened $2$-Stirling permutations that have maximum number of runs. In Section \ref{sec: run-refinement}, we prove the following proposition.

\begin{proposition} \label{max_runs} For integers $n\ge 0, k\ge 0$, 
    \begin{displaymath}
    \vert \flt_{\max}(\stir_n^2) \vert =
		\begin{cases}
            \frac{(3k+1)!}{k!3^k} & \mbox{if } n=3k+1,3k+2\\
            \\
            \frac{(9k+10)}{4}\times \frac{(3k+2)!}{k!3^k} & \mbox{if } n=3k+3
		\end{cases}
	\end{displaymath}
	 
\end{proposition}

\section{Preliminaries and Notation} \label{section: prelim}
We divide the preliminaries into three subsections. The first and second subsections briefly cover the relevant material on non-weighted and weighted labelled structures required in this paper. The third section defines $(c,m)$-colored $B_n$ partitions introduced by Wang in \cite{wang-colored-set} and the modification that needs to be done to make it bijective with flattened $k$-Stirling permutations of order $n$.  
\subsection{Labelled structures}

Much of the terminology in this section is borrowed from \cite{sagan-art-count}. A labelled structure is a function $\mathscr{S}$ that assigns to each finite set $L$ a finite set $\mathscr{S}(L)$ such that
$$\#L=\#M \implies \#\mathscr{S}(L)=\#\mathscr{S}(M).$$

We call $L$ the label set and $\mathscr{S}(L)$ the set of structures on $L$. We let 
$$s_n=\#\mathscr{S}(L)$$
for any $L$ of cardinality $n$, and this is well defined. We have the corresponding egf to a structure
$$\mathcal{F}_{\mathscr{S}}=\displaystyle\sum_{n\ge 0} s_n\frac{z^n}{n!}.$$

To a structure $\mathcal{A}$, the related structure $\mathcal{\overline{A}}$ is defined by 
\begin{displaymath}
\mathcal{\overline{A}}(L) =
	\begin{cases}
		\mathcal{A}(L)  & \mbox{if } L \neq \phi \\
		\phi & \mbox{if } L = \phi.
	\end{cases}
\end{displaymath}

If there are two structures $\mathcal{A,B}$, then the product structure $(\mathcal A\times \mathcal B)$ is the following.
\begin{equation*}
(\mathcal A\times \mathcal B)(L)=\lbrace (A,B) \vert A\in \mathcal A(L_1), B\in \mathcal B(L_2)\ \text{ with }  (L_1,L_2) \text{ is a weak composition of } L \rbrace.
\end{equation*}

Similarly, their composition, $\mathcal{A} \circ \mathcal{B}$ is a structure such that
\begin{multline*}
    (\mathcal{A}\circ\mathcal{B})(L)=\{ (\{B_1,B_2,\dots\},A) \vert \text{ for all } L_1\vert L_2 \vert \dots \vdash L \\
    \text{ with } B_i\in \mathcal{B}(L_i) \text{ for all } i \text{ and } A\in \mathcal{A}(\{ B_1,B_2,\dots \} )  \}
\end{multline*}

We recall two important rules for the labelled structures:
\begin{enumerate}
    \item $\mathcal{F}_{\mathcal{A}\times \mathcal{B}}(z)=\mathcal{F}_{\mathcal{A}}(z)\times \mathcal{F}_{\mathcal{B}}(z)$.
    \item $\mathcal{F}_{\mathcal{A}\circ \mathcal{B}}(z)=\mathcal{F}_{\mathcal{A}}(\mathcal{F}_{\mathcal{\overline{B}}}(z))$.
\end{enumerate}

\subsection{Weighted Species/Weighted labelled structures}
We direct the reader to Section $2.3$ of the book \cite{species-bergeron} by Bergeron, Labelle and Leroux. To keep this paper self-sufficient, we give the relevant material here. Let $\mathbb K \subset \mathbb C$ be an integral domain (for example $\mathbb Z,\mathbb R$ or $\mathbb C$) and $\mathcal{R}$, a ring of formal power series
in an arbitrary number of variables, with coefficients in $\mathbb K$. 
\begin{definition}[{\cite[Definition 2]{species-bergeron}}]
     An $\mathcal{R}$-weighted set is a pair $(A, w)$, where $A$ is a (finite or infinite) set and
$w : A \rightarrow \mathcal R$ is a function which associates a weight $w(\alpha) \in \mathbb R$ with each element $\alpha \in A$. If the
following sum, denoted by $|A|_w$, exists, the weighted set $(A, w)$ is said to be summable and $|A|_w$ is called the inventory (or total weight or cardinality) of
the weighted set $(A, w)$: $|A|_w =\sum_{\alpha\in A} w(\alpha)$.
\end{definition}

\begin{definition}[{\cite[Definition 3]{species-bergeron}}]
    Let $(A, w)$ and $(B, v)$ be $\mathcal R$-weighted sets. A morphism of $\mathcal R$-weighted sets
$f : (A, w) \rightarrow (B, v)$, is a function $f : A \rightarrow B$ compatible with the weighting (one also says that
the function $f$ is weight preserving), that is to say, such that $w = v \circ f$. Moreover, if $f$ is a
bijection, $f$ is called an isomorphism of weighted sets and we write $(A, w) \cong (B, v)$. Observe that
$(A, w) \cong (B, v) \implies |A|_w = |B|_v$ .
\end{definition}

\begin{definition}[{\cite[Definition 6]{species-bergeron}}]
     Let $\mathcal R$ be a ring of formal power series or of polynomials over a ring $\mathbb K \subset \mathbb C$. An
$\mathcal R$-weighted species is a rule $F$ which produces\\
i) for each finite set $U$, a finite or summable $\mathcal R$-weighted set $(F[U], w_U )$,\\
ii) for each bijection $\sigma : U \rightarrow V$ , a function $F[\sigma] : (F[U], w_U )\rightarrow (F[V ], w_V )$ preserving the
weights (i.e., a weighted set morphism).\\
Moreover, the functions $F[\sigma]$ must satisfy the following functoriality properties:\\
a) if $\sigma : U \rightarrow V$ and $\tau : V\rightarrow W$ are bijections, then $F[\tau \circ \sigma] = F[\tau ] \circ F[\sigma]$,\\
b) for each set $U$, if $\mathrm{Id}_U$ denotes the identity bijection of $U$ to $U$, then $F[\mathrm{Id}_U ] = \mathrm{Id}_{F[U]}$.\\
As before, an element $s \in F[U]$ is called an $F$-structure on $U$, and the function $F[\sigma]$, the transport of $F$-structures along $\sigma$.
\end{definition}

\begin{definition}[{\cite[Definition 8]{species-bergeron}}]
   Let $F = F_w$ be an $\mathcal R$-weighted species of structures. The generating series of $F$ is the exponential formal power series $F_w(x)$ with coefficients in $\mathcal R$ defined by
$$F_w(x) = \displaystyle \sum_{n\ge 0} |F[n]|_w\frac{x^n}{n!},$$
where $|F[n]|_w$ is the inventory of the set of F-structures on $[n]$.
\end{definition}

The product and composition structures are defined for weighted structures just like the non-weighted structures.

\begin{proposition}[{\cite[Proposition 11]{species-bergeron}}]
    Let $F_w$ and $G_v$ be two $\mathcal R$-weighted structures. Then
\begin{enumerate}
    \item $(F_w\times G_v)(x)=F_w(x)G_v(x)$
    \item $ (F_w\circ G_v)(x)=F_w(\overline{G_v}(x)) $
\end{enumerate}
\end{proposition}

\subsection{Colored set partitions} 
A \textit{partition} of a set $T$ is a set $\rho$ of nonempty subsets $B_1,B_2,\dots,B_r$ of $T$ such that $\displaystyle T= \uplus_{1\le i\le s}B_i$, which we write as $\rho \vdash T$.  Here, $\uplus$ denotes a disjoint union.
The subsets $B_i$ will be called \textit{blocks}. Sometimes, we write $\rho=B_1\vert\dots\vert B_r$ to mean that $\rho$ is a set partition of $T$ into the blocks $B_1\vert\dots\vert B_r$.

A type $B_n$-partition $\pi$ is a partition of $[-n,n]$ such that if $B$ is a block of $\pi$, then $-B$ is also a block of $\pi$ and $\pi$ contains at most one block such that $B=-B$. If such a block exists, it is called the zero block.    

Wang, in \cite{wang-colored-set}, generalizes type $B_n$- partitions by coloring the elements as opposed to just assigning signs. Let $n,m,c$ be positive integers. Let $\pi$ be a partition of the set $\{0 \}\cup [n]$. Let $C_1$ be a list of $c$ colors and $C_2$ be a list of $m$ colors. For any $x\in \{0 \}\cup [n]$, let $b_x$ be the block of $\pi$ containing $x$. Then, $B_0$ is the zero-block. For any $x\in [n]$, we color $x$ in the first color in $C_2$ if $x$ is the minimum element of the block, some color in $C_1$ if it is not minimum but in the zero block and some color in $C_2$ if it is not minimum and not in the zero block.
Therefore, every element $x$ can be regarded as a pair $(x,cl)$ where $cl$ denotes the color of $x$. Such a partition is called a $(c,m)$-colored $B_n$-partition.
Instead of considering partitions of the set $\{ 0\}\cup [n]$, we will just consider the partitions of $[n]$.

\begin{definition}
    Let $\rho$ be a partition of $[n]$. Let $C=\lbrace c_1,\dots,c_k \rbrace$ be a list of $k$ colors and $w$ be a word on the alphabet $[k]$. We define a \textit{$k$-colored partition} $\rho_w$ of $[n]$ to be a pair $(\rho,w)$ where $x \in [n]$ is colored with the color $c_{w(x)}$, i.e., the word $w$ keeps track of the coloring of $\rho$.  
\end{definition}
For example, if $[3]=\lbrace 1,2,3 \rbrace$ and $\rho=13\vert2$, a $3$-colored partition is $\rho_w=(13\vert2,133)$. We will denote the set of $k$-colored partitions of $[n]$ by $\mathrm{CP}_k(n)$.
Next, we will describe below the \textit{standard block notation} of a $k$-colored partition of $[n]$.

\begin{definition}
     Let $(\rho,w)$ be a $k$-colored partition of $[n]$ where $\rho=B_1\vert\cdots\vert B_r$ and $w=w(1)\dots w(n)$.
    \begin{enumerate}
        \item Relabel the sets $B_1,\dots, B_r$ such that $\min(B_i)<\min(B_j)$ for $1\le i<j\le r$. 
        \item Let $B_i=\lbrace x_{1}<x_{2}<\dots<x_{l} \rbrace$. Then, form the colored word $\pi_i=(x_{1},c_{w(x_1)})\dots(x_{l},c_{w(x_l)})$. Instead of writing a pair, it is convenient to represent $(x_j,c_{w(x_j)})$ by just ${x_j}_{w(x_j)}$. Therefore, the colored word $\pi_i$ will be denoted by ${x_1}_{w(x_1)}\dots {x_l}_{w(x_l)}$.
        \item Juxtapose the colored words $\pi_i$ and place a bar between $\pi_j$ and $\pi_{j+1}$ for $1\le j \le r-1$ and form the barred colored word $$\pi=\pi_1\vert \pi_2\vert \dots \vert \pi_r.$$
        \item This barred colored word $\pi$ is called the standard block notation of $(\rho,w)$.    
    \end{enumerate}
\end{definition}
The subwords $\pi_i$, by abuse of notation, will be called blocks of the $k$-colored partition and from the context, it will be apparent that a block is a colored word when talking about $k$-colored partitions and ordinary sets when talking about set partitions.

\begin{definition}
Let $\rho_w$ be a $k$-colored partition of $[n]$ written in standard block notation. Call $\rho_w$ a \textit{good $k$-colored partition} if it satisfies the following additional rules:
 \begin{enumerate}
     \item[Rule 1:] The minimum element of each block has the color $c_1$ or equivalently, subscript $1$.
     \item[Rule 2:] The elements of the first block, i.e., the block containing the element $``1"$, are colored in one of $\lbrace c_1,\dots, c_{k-1}\rbrace $, or equivalently, the subscript of $x \in B_1$ is bounded above by $k-1$.
 \end{enumerate}
 Define $\mathrm{GCP}_k(n)$ to be the set of good $k$-colored partitions of $[n]$.
 \end{definition}
 \begin{example}
  For example, $\rho_w=1_12_34_2\vert 3_1 \vert5_1$ is a good $4$-colored partition of $[5]$ but not a good $3$-colored partition because the index of the element $``2"$ in the first block is larger than $2$.   
 \end{example}

 \begin{remark}
     Some clarification is required when $k=1$. When $k=1$,  we have $k-1=0$ and so, there is no way to color any elements of the first block other than $``1"$. Therefore, we will assume that the first block only contains $``1"$.
 \end{remark}

\begin{remark}
    The good $k$-colored partitions of $[n]$ are in bijection with $(k-1,k)$-colored type $B_{n-1}$ partitions of Wang. To see this, take a $(k-1,k)$-colored $B_{n-1}$ partition and apply the map $f((x,cl(x)))=(x+1,cl(x))$. The map is a bijection and its image is a good $k$-colored partition of $[n]$. This implies that good $2$-colored partitions are in bijection with the type $B$ partitions of $[-n,n]$.
\end{remark}

\begin{definition}
    A block $B$ of a $k$-indexed partition is called \textit{saturated} if $\vert B\vert=k+1$ and for each $1\le i \le k$, there exists a non-minimum element whose color is $c_i$. 
\end{definition}

\begin{proposition}\label{prop:goodindexnumber} For $n\ge 0$, we have
\begin{equation}
    \vert \mathrm{GCP}_k(n+1) \vert=\sum_{i=0}^n \binom{n}{i} (k-1)^{i}\sum_{r=0}^{n-i} k^{n-i-r} S(n-i,r)
\end{equation} where $S(a,b)$ denotes the Stirling number of the second kind with $S(a,0)=0$ for all $a\ge 0$.
\end{proposition}
\begin{proof}
    Suppose the first block has $i+1$ elements, then we can choose the $i$ elements other than the element $``1"$ in $\binom{n}{i}$ ways. These $i$ elements can be colored in $(k-1)^i$ ways.
    The remaining $n-i$ elements can be partitioned into $r$ parts in $S(n-i,r)$ ways. Since the elements other than the minimum elements in each block can be colored in $k$ ways, we have the factor $k^{n-i-r}$. Summing these appropriately gives us our equation.  
\end{proof}
Notice that RHS of \eqref{prop:goodindexnumber} is the same as the RHS of \eqref{thm:bij_rec}, which brings us to the following result.
\begin{theorem}\label{thm:bijective}
    There is a bijection from $\flt(\stir_n^k)$, the set of flattened $k$-Stirling permutations to the set, $\mathrm{GCP}_k(n)$, of good $k$-colored partitions of $[n]$.
\end{theorem}
In the next section, we construct a bijection and prove Theorem \ref{thm:bijective} which will, in turn, complete the proof of Theorem \ref{thm:bij_rec}.

\section{The bijection} \label{sec: bijection}
We shall describe the bijection $\Phi$ from $\mathrm{GCP}_k(n)$ to the set $\flt(\stir_n^k)$. This will be done in the following way. If $(\rho,w)=\pi_1\vert\pi_2\vert\dots\vert\pi_r$ is a good $k$-colored partition of $[n]$ written in standard block notation, we will obtain a word $SW_i$ from each block $\pi_i$. The juxtaposition of all these words $W=SW_1SW_2\dots SW_r$ will be the image of $\rho_w$ under $\Phi$, (i.e., $\Phi(\rho_w)=W$). Of course, we will have to prove that the image is indeed a flattened $k$-Stirling permutation of order $n$ and the map is bijective.
\begin{definition} \label{def: bij}
    Let $(\rho,w)=\pi_1\vert\pi_2\vert\dots\vert\pi_r$ be a good $k$-colored partition of $[n]$ written in standard block notation. The map $\Phi: \mathrm{GCP}_k(n) \rightarrow \flt(\stir_n^k)$ is as follows.
    
    Let $\pi_i={x_1}_{w(x_1)}\dots{x_l}_{w(x_l)}$ be a block of $\rho_w$ in its standard block notation. Since $\pi_i$ is a block in the standard block notation, we have $x_1<x_2<\dots<x_l$. 

Follow the steps described below to obtain a word $SW_i$ from $\pi_i$.
\begin{enumerate}
    \item Place $k$ copies of the minimum element $x_1$ in a line with gaps between them (represented by bars) in the following manner. The gaps are to be numbered from right to left. Label the copies of $x_1$ from right to left by adding a superscript.

$$\underbrace{\vert}_{\text{gap $k$}}x_1^{(k)}\underbrace{\vert }_{\text{gap}} \dots \underbrace{\vert }_{\text{gap}} x_1^{(3)} \underbrace{\vert}_{\text{gap $2$}} x_1^{(2)}\underbrace{\vert }_{\text{gap $1$}} x_1^{(1)}$$
    \item Set $i\leftarrow 2$.
    \item Set $x\leftarrow x_i$.
    \item Since $x$ has the subscript $w(x)$ (it has color $c_{w(x)}$), place the $k$ copies of $x$ in the gap $w(x)$. Place them immediately to the left of the bar (to the right of any letters that already exist to the left of the bar), as shown below.
    $$\underbrace{\vert}_{\text{gap $k$}}x_1^{(k)}\underbrace{\vert }_{\text{gap}} \dots \underbrace{\vert}_{\text{gap}}x_1^{(w(x)+1)} \underbrace{\dots x^k\vert}_{\text{gap $w(x)$}}  x_1^{(w(x))} \underbrace{\vert}_{\text{gap}} \dots \underbrace{\vert }_{\text{gap}} x_1^{(2)}\underbrace{\vert }_{\text{gap $1$}} x_1^{(1)}$$ 
    \item Set $i\leftarrow i+1$. If $i\le l$, return to Step 3. Else, proceed to next step.

    \item The word created, after placing all the letters $x_2,\dots, x_l$ and removing the bars, is our required word $SW_i$.
\end{enumerate}

Define $\Phi((\rho,w)):=SW_1SW_2\dots SW_r$ to be the juxtaposition of the words $SW_i$ obtained from the blocks $\pi_i$ for $1\le i\le r$.
\end{definition}

We give an example to illustrate.
\begin{example}
    We can see that $(\rho,w)=1_12_34_26_3\vert3_1\vert5_1$ is a good $4$-colored partition of $[6]$ written in standard block notation. The blocks of $\rho$ are $B_1=\lbrace 1,2,4,6 \rbrace$, $B_2=\{ 3 \}$ and $B_3=\{ 5\}$ and the associated word $w$ is $131213$.
    To form $SW_1$, we will place the four copies of $``1"$ in the line with gaps as mentioned.
    $$\underbrace{\vert}_{4} 1 \underbrace{\vert}_{3} 1 \underbrace{\vert}_{2} 1 \underbrace{\vert}_{1} 1$$
    Now, we will place the four copies of $``2"$ before the bar $3$.
    $$\underbrace{\vert}_{4} 1\underbrace{2222 \vert}_{3} 1 \underbrace{\vert}_{2} 1 \underbrace{\vert}_{1} 1$$
    Now, we place the four copies of $``4"$ before the bar $2$.
    $$\underbrace{\vert}_{4} 1 \underbrace{2222\vert}_{3} 1 \underbrace{4444\vert}_{2} 1 \underbrace{\vert}_{1} 1$$
    Finally, we place the four copies of $``6"$ before the bar $3$.
    $$\underbrace{\vert}_{4} 1 \underbrace{22226666\vert}_{3} 1 \underbrace{4444\vert}_{2} 1 \underbrace{\vert}_{1} 1$$
    We will remove the bars and end up with the word $SW_1=1222266661444411$
    Similarly, $SW_2$ will be $3333$ and $SW_3$ will be $5555$.
    Therefore, the flattened $4$- Stirling permutation corresponding to $1_12_34_26_3\vert3_1\vert5_1$ is $122226666144441133335555$.
\end{example}
We prove that the algorithm described above produces a flattened $k$-Stirling permutation of order $n$ from a good $k$- colored partition of $[n]$.
\begin{lemma}
    If $\rho_w\in \mathrm{GCP}_k(n)$ and $\Phi$ is the map described above, then $\Phi(\rho_w)\in \flt(\stir_n^k)$.
\end{lemma}
\begin{proof}
    First and foremost, we have to show that $\Phi(\rho_w)\in \stir^k_n$. By construction, $\Phi(\rho_w)=SW_1SW_2\dots SW_r$. In each subword $SW_i$, all $k$ copies of every letter other than the minimum letter in that subword appears together. For the minimum letter, since it is minimum in that subword, every letter between two occurrences of it will be larger than it. Therefore, $\Phi(\rho_w)\in \stir_n^k$.
    
    Next, we show that the image is a flattened $k$-Stirling permutation. It is clear that in a subword $SW_i$, the last letter of the subword is the minimum letter of the subword. Therefore, every run in the subword $SW_{i+1}$ either is continuation of a run starting with the minimum letter of $SW_i$ or is a new run started with the minimum letter of $SW_{i+1}$. Since $\min(B_i)<\min(B_j)$ for $i<j$, in standard block notation, the leading terms of the runs are weakly increasing. This completes the proof.
\end{proof}

Now, we prove that the map is indeed a bijection on the two sets $\mathrm{GCP}_k(n)$ and $\flt(\stir_n^k)$.
\begin{lemma}
    The map $\Phi$, described above, is a bijection.
\end{lemma}
\begin{proof}
\textit
    We shall show that an inverse exists under this map and that it is unique. Let $W \in \flt(\stir_n^k)$. 
    \begin{enumerate}
        \item[Step 1:] Set $W'\leftarrow W$ and $j\leftarrow 1$. 
        \item[Step 2:]  Look at the last letter $\alpha$ of $W'$. Let $i$ be the position of the rightmost occurrence of a letter strictly smaller than $\alpha$. Define $SW_j:=W'_{i+1}W'_{i+2}\dots W'_{l(W')}$ (where $l(W')$ is the length of $W'$). By our construction of the bijection, this subword came from a block $B_j$ whose minimum element is $\alpha$ and $B_j$ contains the letters that appear in $SW_j$.  If we number the copies of $\alpha$ in $SW_j$ from right to left, it is clear how to color the elements of $B_j$. The minimum element $\alpha$ will have subscript $1$ and if the copies of the letter $x$ lie between the $i^{th}$ and $(i+1)^{th}$ copy of $\alpha$ (resp. after the $k^{th}$ copy of $\alpha$), then it will have subscript $w(x)=i$ (resp. $w(x)=k$).
       \item[Step 3:] Delete the subword $SW_j$ from $W'$ and set $W'\leftarrow W'\setminus SW_j$. If $l(W')>0$, set $j\leftarrow j+1$ and go back to Step 2. Proceed if $l(W')=0$.
       \item[Step 4:] We have $W=SW_j\dots SW_1$ and by Step 2, that it is the image of the unique good $k$-colored partition $\rho_w$ where $\rho=B_j\vert \dots \vert B_1$ and $w$ is the word formed by the subscripts of the corresponding elements. Furthermore, the sets $B_i$ satisfy $\min(B_i)<\min(B_j)$ for $j<i$.
    \end{enumerate}
\end{proof}
Therefore, we have completed the proof of Theorem $\eqref{thm:bijective}$ and thereby, of Theorem $\eqref{thm:bij_rec}$.

\section{The proof of Theorem \eqref{thm:egf} and some corollaries} \label{section: proof-corollary}
\begin{proof}[Proof of Theorem \ref{thm:egf}]
    We define a labelled structure $\mathrm{COLOR}_r(L)$ that assigns one of $r$ colors to the elements of the label set $L$. The egf of the structure is $\mathcal{F}_{\mathrm{COLOR}_r}(z):=\sum_{n\ge 0} c_n\frac{z^n}{n!}=\sum_{n\ge 0} r^n\frac{z^n}{n!}=e^{rz}$.

We define a labelled structure $\mathrm{RCOLOR}_r(L)$ that assigns one of the $r$ colors to the every element of $L$ other than the element with the smallest label in $L$. The egf of the structure is $$\mathcal{F}_{\mathrm{RCOLOR}_r}(z):=\sum_{n\ge 0} c_n\frac{z^n}{n!}=\sum_{n\ge 0} r^{n-1}\frac{z^n}{n!}=1+\frac{e^{rz}-1}{r}.$$

Let $\mathrm{G}_k$ be the labelled structure that is obtained by doing the following.
Given the label set $L=\{ l_1,\dots, l_n\}$, then make the weak composition $(L_1,L_2)$ where $L_1$ corresponds to the elements in the first block of a good $k$-colored partition and $L_2=L\backslash L_1$, its complement. On $L_1$, we want to color the elements using $k-1$ colors. On $L_2$, we want to further partition $L_2$ into blocks and in each block, we want to color all but the non-minimum element in one of $k$ colors. 

In the language of labelled structures, the structure $$\mathrm{G}_k=\mathrm{COLOR}_{k-1}\times \mathrm{SET}(\mathrm{\overline{{RCOLOR}_{k}}}).$$
Therefore, the egf $$\mathcal{F}_{\mathrm{G}_k}(z)=\exp{\bigg((k-1)z + \frac{e^{kz}-1}{k}\bigg)}$$ follows from the product and composition rule.

If we notice that $\# \mathrm{G}_k(\{2,\dots,n+1\})=\vert \mathrm{GCP}_{k}(n+1) \vert =\vert \flt(\stir^k_{n+1}) \vert$, then we have proved \ref{thm:egf}. 
\end{proof}

Using this, we are in a position to prove the following corollaries.

\begin{proof}[Proof of Corollary \ref{thm: series-flat}]
We can write $e^{\frac{e^{kz}}{k}}$ as the series $$\sum_{i\ge 0} \frac{e^{kzi}}{k^ii!}=\sum_{i\ge 0 }\sum_{j\ge 0} \frac{(kiz)^j}{k^ii!j!}.$$
The coefficient of $\frac{z^i}{i!}$ in the series is $$\sum_{r\ge 0} \frac{(kr)^{i}}{k^rr!}.$$
\begin{eqnarray*}
    \vert \flt(\stir^k_{n+1}) \vert &=&\bigg[\frac{z^{n}}{n!}\bigg]\exp{\bigg((k-1)z + \frac{e^{kz}-1}{k}\bigg)}\\
    &=&\bigg[\frac{z^{n}}{n!}\bigg]e^{-1/k}\exp\Bigg((k-1)z+\frac{e^{kz}}{k} \Bigg)\\
  &=&e^{-1/k}\sum_{i=0}^{n} \binom{n}{i} \sum_{r\ge 0} \frac{(kr)^i}{k^rr!}(k-1)^{n-i}\\
  &=&e^{-1/k} \sum_{r \ge 0} \frac{(kr+k-1)^{n}}{k^rr!}.
\end{eqnarray*}
\end{proof}

The bijection allows us to prove the recurrence in Theorem \ref{thm: recurrence}.

\begin{proof}[Proof of Theorem \ref{thm: recurrence}]
 Suppose the block that contains the element $``n+1"$ does not contain $``1"$ and has $r$ elements, it can be chosen and colored in $\binom{n-1}{r-1}k^{r-1}$ ways. The other blocks can be standardised (replacing the smallest number by $1$, the second smallest by $2$ and so on), while retaining the colors, to a give a good $k$-colored partition of $[n-r+1]$. This process is invertible. Once the block $B$ containing $``n+1"$ is known, we can replace each letter in $[n+1-|B|]$ with letters of $[n+1]-B$  while retaining their colors to get back the original good $k$-colored partition of $[n+1]$.
 
 Suppose it is in the block with $``1"$, then it could have any of the $k-1$ colors. If $``n+1"$ were removed, the remaining elements would form a good $k$-colored partition of $[n]$. Since the number of good $k$-colored partitions of $[n]$ is the number of flattened $k$-Stirling permutations of order $n$ and this process is invertible, we are done.
 
\end{proof}
\section{A refinement of the set of flattened $k$-Stirling permutations}\label{sec: run-refinement}
Recall that runs of a flattened $k$-Stirling permutation $W$ are maximal contiguous weakly increasing subwords of $W$. If we define $\run(W)$ to be the number of runs in $W$, then we can refine the set $\flt(\stir_n^k)$ based on the number of runs. For integers $n\ge 1$ and $k\ge 1$, we will denote ,by $\flt_s(\stir_n^k)$, the set of flattened $k$-Stirling permutations of order $n$ with $s$ runs.
If we define $\des(W)$ to be number of indices $i$ with $1\le i<nk$ where $W(i)>W(i+1)$, then one can see that $\run(W)=\des(W)+1$. Clearly, $1^k2^k\dots n^k$ is the only flattened $k$-Stirling permutation that has $1$ run.
In the following proposition, we determine the maximum number of runs that a flattened $k$-Stirling permutation has.
\begin{proposition}
    If $W$ is a flattened $k$-Stirling permutation, then  $1 \le \run(W)\le \lceil \frac{kn}{k+1} \rceil$.
\end{proposition}
\begin{proof}
    Every run in a flattened $k$-Stirling permutation contains $k$ copies of at least one element (except possibly for the last run). Therefore, the length of each run is at least $k+1$ (except possibly for the last run). This gives us the bound $\run(W)\le \lfloor \frac{kn}{k+1} \rfloor +1$ when $k+1$ does not divide $n$ and $\run(W)\le \frac{kn}{k+1}$ when $k+1$ divides $n$.
    It is easy to construct a flattened $k$-Stirling permutation whose runs are all of length $k+1$ (except possibly the last run), completing our proof.
\end{proof}

We move on to the proof of Proposition \ref{prop: 2_runs}.

\begin{proof}[Proof of Proposition \ref{prop: 2_runs}]
     For a good $k$-colored partition $\rho_w=(\rho,w)$, singleton blocks do not create descents under the action of $\Phi$. Thus, there is only one non-singleton block in the standard block notation of $\rho_w$. Furthermore, all the non-minimum elements in the non-singleton block must have the same subscript. Otherwise, the subword formed by that block will have a descent. We split the counting into two cases.
\begin{enumerate}
    \item[Case 1:]{(The non-singleton contains $``1"$)}\\
    For $1\le i\le n-1$, we can pick the other $i$ elements in the block in $\binom{n-1}{i}$ ways and their subscripts in $k-1$ ways. Summing over $i$, this produces $(k-1)(2^{n-1}-1)$ flattened $k$-Stirling permutations.
    \item[Case 2:]{(The non-singleton does not contain $``1"$)}\\
    For $2\le i \le n-1$, we can pick the non-singleton block in $\binom{n-1}{i}$ ways and their subscripts in $k$ ways. Summing over $i$, this produces $k(2^{n-1}-1-(n-1))$ flattened $k$-Stirling permutations.
\end{enumerate}
     Summing these two numbers finishes the proof.
\end{proof}

We now prove \ref{prop:3_runs}. The argument is similar to the one in Proposition \ref{prop: 2_runs} but is slightly more involved.

\begin{proof}[Proof of {Proposition \ref{prop:3_runs}}]
    In light of the formula $\run(W)=\des(W)+1$, we need to have $2$ descents in our flattened $k$-Stirling permutation. Therefore, there are at most $2$ non-singleton blocks in the standard block notation of its associated good $k$-colored notation.
    Here, we split into following cases.
    \begin{enumerate}
        \item[Case 1:] (There is one non-singleton block and it contains $``1"$)\\
        If there is only one non-singleton block, then there are $2$ different subscripts among the non-minimum elements. Choose the elements in the block in $\binom{n-1}{i}$ ways and then, from the chosen ones, pick out $j$ elements to have one of the subscripts in $\binom{i}{j}$ ways. This forces the other elements to have the other subscript. The subscript can be assigned in $\frac{(k-1)(k-2)}{2}$. Summing over $i,j$, we get $\frac{(k-1)(k-2)}{2}\sum_{i=2}^{n-1}\binom{n-1}{i}\sum_{j=1}^{i-1}\binom{i}{j}$.
        \item[Case 2:] (There is one non-singleton block and it doesn't contain $``1"$)\\
        Choose the elements in the block in $\binom{n-1}{i}$ ways and then, from the chosen ones, pick out $j$ elements to have one of the subscripts in $\binom{i-1}{j}$ ways from the elements that are not minimum. We have to pick at least $1$ element and at most $i-2$, as all of them cannot have the same subscript. This forces the unpicked elements to have the other subscript. The subscript can be assigned in $\frac{k(k-1)}{2}$. Summing over $i,j$ gives $\frac{k(k-1)}{2}\sum_{i=3}^{n-1}\binom{n-1}{i}\sum_{j=1}^{i-2}\binom{i-1}{j}$.
        \item[Case 3:] (There are two non-singleton blocks and one of them contains $``1"$)\\
        Choose the elements in the block containing $``1"$ in $\binom{n-1}{i}$ ways and then, assign subscripts in $k-1$ ways. Choose the $j$ elements of the second block from the remaining elements in $\binom{n-1-i}{j}$ ways and assign subscripts in $k$ ways. Summing over $i,j$ gives $k(k-1)\sum_{i=1}^{n-3}\binom{n-1}{i}\sum_{j=2}^{n-1-i}\binom{n-1-i}{j}$ (In the first summation, $i$ goes only till $n-3$ because there needs to be a non-singleton block of size $2$).
        \item[Case 4:] (There are two non-singleton blocks and neither of them contain $``1"$)\\
        Choose the elements in one block in $\binom{n-1}{i}$ ways and then, assign subscripts in $k$ ways. Choose the $j$ elements of the second block from the remaining elements in $\binom{n-1-i}{j}$ ways and assign subscripts in $k$ ways. Since the choice is unordered, we have to divide by $2$. Summing over $i,j$, we get $\frac{k^2}{2}\sum_{i=2}^{n-3}\binom{n-1}{i}\sum_{j=2}^{n-1-i}\binom{n-1-i}{j}$.
    \end{enumerate}
Summing these gives
\begin{eqnarray*}
    \vert \flt_3(\stir_n^k) \vert&=&  \frac{(k-1)(k-2)}{2}\sum_{i=2}^{n-1}\binom{n-1}{i}\sum_{j=1}^{i-1}\binom{i}{j}  +  \frac{k(k-1)}{2}\sum_{i=3}^{n-1}\binom{n-1}{i}\sum_{j=1}^{i-2}\binom{i-1}{j}\\  &+&  k(k-1)\sum_{i=1}^{n-3}\binom{n-1}{i}\sum_{j=2}^{n-1-i}\binom{n-1-i}{j}  +  \frac{k^2}{2}\sum_{i=2}^{n-3}\binom{n-1}{i}\sum_{j=2}^{n-1-i}\binom{n-1-i}{j}\\
    &=& \frac{(k-1)(k-2)}{2}\sum_{i=2}^{n-1}\binom{n-1}{i}(2^i-2)  +  \frac{k(k-1)}{2}\sum_{i=3}^{n-1}\binom{n-1}{i}(2^{i-1}-2)\\  &+&  k(k-1)\sum_{i=1}^{n-3}\binom{n-1}{i}\sum_{j=2}^{n-1-i}\binom{n-1-i}{j}  +  \frac{k^2}{2}\sum_{i=2}^{n-3}\binom{n-1}{i}\sum_{j=2}^{n-1-i}\binom{n-1-i}{j}. \\
    &=& \frac{(k-1)(k-2)}{2}(3^{n-1}-2^n+1)  +  \frac{k(k-1)}{12}(3^n-6\times 2^{n}+6n+3)\\  &+&  k(k-1)\sum_{i=1}^{n-3}\binom{n-1}{i}\sum_{j=2}^{n-1-i}\binom{n-1-i}{j}  +  \frac{k^2}{2}\sum_{i=2}^{n-3}\binom{n-1}{i}\sum_{j=2}^{n-1-i}\binom{n-1-i}{j}
\end{eqnarray*}
\end{proof}

We prove the formula for the number of flattened $2$-Stirling permutations with maximum number of runs.

\begin{proof}[Proof of Proposition \ref{max_runs}]
     Let $W$ be a flattened $2$-Stirling permutation of order $n$ and $\Phi(W)=\rho_w$ be its associated type $B$ set partition. Suppose there are $r$ blocks in $\rho_w$, say $B_1,B_2,\dots, B_r$, then $\vert B_1\vert+\dots+\vert B_r\vert=n$. The number of descents of $W$ is $\run(W)-1$. Each block contributes only $0,1,2$ descents in the subword it creates. The maximum descents a block can create is $2$ and that happens only if the block contains a saturated $3$-set. The first block can make at most $1$ descent. Every other block can produce at most $2$. 
     
     In a good $k$-partition that maps to flattened $k$-Stirling permutation with maximum descents, there can be at most one block $B_i$ such that $3\le |B_i|<5$. If not, we can remove two elements from this block and color the remainder to make it saturated and put the removed elements in a separate block and have more descents in the image, contradicting maximality of descents. Therefore, most of the elements have to be partitioned into saturated $3$-sets.

     Similarly, outside of the first block, there can be at most $2$ doubleton blocks as $3$ blocks can be split into two saturated $3$-blocks and contribute more descents, contradicting maximality. Further, all blocks of size greater than $3$ contain a saturated $3$-subset as otherwise, we can adjust colors to get more descents. We next split the proof based on the value of $n \! \!\pmod 3$.

     \begin{enumerate}
         \item If $n=3k+1$, then the number of descents must be $2k$. 
         Let $r_3$ be the number of blocks that contribute $2$ descents, $r_2$ the number of blocks that contribute $1$ and $r_1$ that contribute no descents, then the following hold: 
         $$2r_3+r_2=2k$$
         $$3r_3+2r_2+r_1\le 3k+1$$
         We have $$2r_1+r_2 \le 2.$$
         If there is only one doubleton, then we would be left with $3k-1$ elements to get $2k-1$ descents. If we put $3k-3$ in saturated $3$-sets, we would get $2k-2$ descents. This forces another doubleton to appear. Therefore, there are either $2$ doubletons or $1$ singleton. The rest are all saturated $3$-subsets because either $r_2=2$ and we have to produce $2k-2$ descents with $3k-3$ elements or $r_1=1$ and we have to produce $2k$ descents with $3k$ elements.         
         
     When there is one singleton, all the other elements form saturated $3$-subsets. This can be done in $\frac{2^k}{k!}\binom{3k}{\underbrace{3,3,\dots,3,3}_{k \text{ times}}}$ ways.
     When there are two doubletons, the first block must be a doubleton. This can be chosen in $n-1$ ways. The next doubleton block can be chosen and colored in $2\times \binom{n-2}{2}$ ways. The other size $3$ blocks can be formed in $\frac{2^{k-1}}{(k-1)!}\binom{3(k-1)}{\underbrace{3,3,\dots,3,3}_{k-1 \text{ times}}}=\frac{3\times2^{k}}{(k-1)!}\binom{3k}{\underbrace{3,3,\dots,3,3}_{k \text{ times}}}$ ways.     
     This gives a total of $$\frac{2^k}{k!}\binom{3k}{\underbrace{3,3,\dots,3,3}_{k \text{ times}}}+\frac{3\times2^{k}}{(k-1)!}\binom{3k}{\underbrace{3,3,\dots,3,3}_{k \text{ times}}}=\frac{2^k}{k!}(3k+1)\binom{3k}{\underbrace{3,3,\dots,3,3}_{k \text{ times}}}. $$
     \item If $n=3k+2$, then $2r_1+r_2 \le 1$. This implies there is only $1$ doubleton. The doubleton block has to be the first block and can be chosen in $n-1$ ways. The rest are saturated $3$-subsets (by the argument in the first case) which can be chosen in $\frac{2^k}{k!}(3k+1)\binom{3k}{\underbrace{3,3,\dots,3,3}_{k \text{ times}}}$ ways.
     \item If $n=3k+3$, then $2r_1+r_2 \le 3$. We need to get $2k+1$ descents. When $2r_1+r_2=0$, this means the first block $B_1$ has $|B_1|\ge 3$. This forces all other blocks to be saturated $3$-sets which will give at most $2k$ descents, which leaves the first block having size $3$.
     
     When $2r_1+r_2=1$, there is only one doubleton. If there is no singleton and the doubleton is the first block, then there is a block of size $4$ that contains a saturated $3$-subset and the others are saturated $3$-sets. Suppose the doubleton is not the first block, the first block has at least $3$ elements. It can contribute at most $1$ descent, the doubleton contributes at most $1$ descent and the remaining $2k-1$ descents has to come from at most $3k-2$ elements, which is not possible.
     
     When $2r_1+r_2=2$, there is either one singleton or two doubletons. If there is one singleton and no doubletons, we would have $3k+2$ elements to produce $2k+1$ descents and this is not possible without a doubleton. If there are two doubletons, then this would account for $2$ descents. The other $2k-1$ descents should come from $3k-1$ elements. This implies the presence of another doubleton. Therefore, there is a singleton and a doubleton or three doubletons.

     When $2r_1+r_2=3$, there are no new cases.

      Now, if the first block has $3$ elements, then we can make that choice in $\binom{n-1}{2}$ ways. The rest can be made into saturated $3$-sets in $\frac{2^{k}}{k!}\binom{3k}{\underbrace{3,3,\dots,3,3}_{k \text{ times}}}$ ways. 
     $$I_1=\binom{n-1}{2}\frac{2^{k}}{k!}\binom{3k}{\underbrace{3,3,\dots,3,3}_{k \text{ times}}}=\frac{2^{k-1}}{k!}(3k+2)(3k+1)\binom{3k}{\underbrace{3,3,\dots,3,3}_{k \text{ times}}}.$$

      For the case when one of the blocks has size $4$, first, we pick out the element in the first block in $n-1$ ways. We, then, pick the $4$-set and index it in $6\binom{n-2}{4}$ ways. Finally, we make saturated $3$-sets from the remaining elements in $\frac{2^{k-1}}{k-1!}\binom{3k-3}{\underbrace{3,3,\dots,3,3}_{k-1 \text{ times}}}$ ways.
     $$I_2= 6(n-1)\binom{n-2}{4}\frac{2^{k-1}}{k-1!}\binom{3k-3}{\underbrace{3,3,\dots,3,3}_{k-1 \text{ times}}}=3(3k+2)(3k+1)\frac{2^{k-2}}{k-1!}\binom{3k}{\underbrace{3,3,\dots,3,3}_{k \text{ times}}}.$$
     
     If the singleton is the first block, then the doubleton can be chosen and colored in $2\binom{n-1}{2}$ ways. The remaining $n-3$ elements can be made into saturated $3$-sets in $\frac{2^{k}}{k!}\binom{3k}{\underbrace{3,3,\dots,3,3}_{k \text{ times}}}$ ways. Thus,
     $$I_3= 2\binom{n-1}{2} \frac{2^{k}}{k!}\binom{3k}{\underbrace{3,3,\dots,3,3}_{k \text{ times}}}= \frac{2^{k}}{k!}(3k+2)(3k+1)\binom{3k}{\underbrace{3,3,\dots,3,3}_{k \text{ times}}}.$$
     When the doubleton is the first block, this choice can be made in $n-1=3k+2$ ways. From the remaining $3k+1$ numbers, we make saturated $3$-subsets in $\frac{2^{k}}{k!}(3k+2)\binom{3k+1}{\underbrace{3,3,\dots,3,3}_{k \text{ times}}}$ ways. The remaining element after the choice will be singleton. Thus, this gives a count of
     $$I_4=\frac{2^{k}}{k!}(3k+2)(3k+1)\binom{3k}{\underbrace{3,3,\dots,3,3}_{k \text{ times}}}.$$ 
     If there are 3 doubletons in the standard block notation, then we can pick the first block in $n-1$ ways, the second doubleton in $2\binom{n-2}{2}$ ways and third in $2\binom{n-4}{2}$. The rest can be made into saturated $3$-sets in $\frac{2^{k-1}}{k-1!}\binom{3k-3}{\underbrace{3,3,\dots,3,3}_{k-1 \text{ times}}}$ ways. However, the order of picking the doubletons does not matter and we have to divide by a factor of two. Thus, this gives a count of 
     $$I_5=3(3k+2)(3k+1)\frac{2^{k-1}}{k-1!}\binom{3k}{\underbrace{3,3,\dots,3,3}_{k \text{ times}}}.$$
          The answer we require is $\sum_{i=1}^5 I_i$. Clearly, this equals the RHS given in Proposition \ref{max_runs}.
     \end{enumerate}

\end{proof}
\section{EGF for the descent statistic over flattened $k$-Stirling permutations} \label{section: descent-egf}
We need the following lemma to connect descents to the colors in the set partition before we prove Theorem \ref{thm: main-egf}.
\begin{lemma}\label{lemma: descent-color}
    Let $(\rho,w)=\pi_1 \vert \pi_2 \vert \dots \vert \pi_r$ be a good $k$-colored partition of $[n]$, written in standard block notation. If $SW_i$ created by $\pi_i={x_1}_{w(x_1)}\dots{x_l}_{w(x_l)}$, under the map $\Phi$, then $\des(SW_i)=\#\lbrace w(x_2),\dots,w(x_l) \rbrace$, i.e., the number of descents in the word $SW_i$ is the number of distinct colors that appear in the coloring of the non-minimum elements of $\pi_i$.
\end{lemma}
\begin{proof}
    Notice that in the subword $SW_i$, between two consecutive occurrences of $x_1^{(j+1)}$ and $x_1^{(j)}$, there is at most one descent and this descent occurs if there is an element of the color $c_j$ in the block $\pi_i$ and the descent would occur just before $x_1^{(j)}$.  
\end{proof} 
We now move on to our proof of Theorem \ref{thm: main-egf}.
\begin{proof}[Proof of Theorem \ref{thm: main-egf}]
    In light of Lemma \ref{lemma: descent-color}, we want to keep track of the number of distinct colors of the non-minimum element in each block if we want to keep track of descents over flattened $k$-Stirling permutations.
  
The structure $\mathrm{WEIGHT}$ on the label set $L$ is a structure that gives the weight $t$ to a non-empty label set and $1$ to an empty label set. The egf is then $$\mathcal{F}_{\mathrm{WEIGHT}}(z)=t(e^z-1)+1.$$

We define a labelled structure $\mathrm{DCOL}_r$ on the label set $L$. This structure splits the set $L$ into a weak composition $(L_1,\dots, L_r)$ where $L_i$ contains the elements of color $c_i$ for $1\le i\le r$. On each block, it assigns weight $t$ if $L_i$ is non-empty and $1$ if it is empty. Therefore, in the language of labelled structures, $$\mathrm{DCOL}_r(L)=\mathrm{WEIGHT}^r(L).$$  

We will define a labelled structure $\mathrm{RDCOL}_r$ on the label set $L=\{ l_1,l_2,\dots\}$. This structure is similar to $\mathrm{DCOL}_r$ but splits the set $L\backslash\{l_1\}$ into a weak composition $(L_1,\dots, L_r)$ where $L_i$ contains the elements of color $c_i$ for $1\le i\le r$. On each block, it assigns weight $t$ if $L_i$ is non-empty and $1$ if it is empty.

We will find the egf $\mathcal{F}_{\mathrm{RDCOL}_r}$  of this labelled structure directly. Let $R_n$ be the set of ordered tuples of possibly empty sets $(B_1,\dots, B_k)$ such that the $B_i$ are pairwise disjoint and  $\uplus B_i =\{l_2,\dots,l_n\}$. For $\rho\in R_n$, let $\mathrm{ne}(\rho)$ be the number of non-empty blocks of $\rho$. We are interested in $$\sum_{n\ge 1}\sum_{\rho \in R_n}t^{\mathrm{ne}(\rho)}\frac{z^n}{n!}.$$  

This can be enumerated as follows. When $n=1$, there are no ways to color and therefore, we get $t^0z$. When $n\ge 2, 1\le j\le k$, partition $\{l_2,\dots,l_n\}$ into $j$ blocks in $S(n-1,j)$ ways. Choose the colors that these blocks have in $\binom{k}{j}$ ways and match the colors to the blocks in $j!$ ways. Thus, the egf is  $$\mathcal{F}_{\mathrm{RDCOL}_k}(z)=z + \sum_{n\ge 2} \sum_{j=1}^{k} \frac{k!}{(k-j)!}S(n-1,j)t^j\frac{z^n}{n!} $$

Interchanging the summation, we get 
\begin{equation}
  \mathcal{F}_{\mathrm{RDCOL}_k}(z)=z+\displaystyle \sum_{j=1}^{k}\binom{k}{j}j!t^j  \sum_{n=2}^{\infty}S(n-1,j) \frac{z^n}{n!}. 
\end{equation}

If we set $H_j(z):=\displaystyle \sum_{n=0}^{\infty} S(n-1,j)\frac{z^n}{n!}$, we can obtain closed forms for this series for small values of $j$ using the explicit formula $$\displaystyle S(n,k) ={\frac {k^{n}}{k!}}-\sum _{r=1}^{k-1}\frac {S(n,r)}{(k-r)!}.$$

For $j=1,2$, we have
\begin{eqnarray}
  H_1(z)&=&e^z-z-1\\
  H_2(z)&=&\frac{1}{4} (3 +2 z - 4 e^z + e^{2 z})
\end{eqnarray}

Setting $k=1$, we get 
\begin{eqnarray}
    \mathcal{F}_{\mathrm{RDCOL}_1}(z)&=&z+t(e^z-z-1)\\
    \mathcal{F}_{\mathrm{RDCOL}_2}(z)&=&z+2t(e^z-z-1)+2t^2\bigg(\frac{3 +2 z - 4 e^z + e^{2 z}}{4}\bigg)
\end{eqnarray}

We define the labelled structure $\mathrm{DESGCP}_r$ that does the following. Given the label set $L=\lbrace 1,\dots,n \rbrace$, then we form the weak composition $(L_1,L_2)$ where $L_1$ contains the elements of the first block and $L_2$ contains the rest of the elements. We want to split $L_1$ into an ordered weak composition $(L'_{1},\dots, L'_{r-1})$ where the set $L'_i$ contains the elements of color $c_i$. To each $L'_i$, we give the weight $t$ if it is non-empty and $1$ if empty. 
On $L_2$, we form a partition $B\vdash L_2$. If $B=B_1\vert B_2 \vert \dots$, then on each block $B_i$, we further break the non-minimal elements into $k$ possibly empty blocks $B_{i,1},\dots,B_{i,r}$. On non-empty blocks, we assign the weight $t$ and on empty blocks, we assign the weight $1$. In the language of labelled structures, we have $$\mathrm{DESGCP}_r=\mathrm{DCOL}_{r-1} \times \mathrm{SET}(\overline{\mathrm{RDCOL}_r}).$$

We note that $$A^{\flt}_{n+1,r}(t)=\displaystyle \sum_{\pi \in \mathrm{GCP}_r(n+1)} \sum_{B \text{ is a block of } \pi} t^{\mathrm{ne(B)}}=\mathrm{weight}(\mathrm{DESGCP}_r(\{2,\dots,n+1\})).$$ Theorem \ref{thm: main-egf}, Corollary \ref{thm:naba},Corollary \ref{thm:stirling} follow from this. 
\end{proof}

\section{Concluding remarks and questions}
Below, we list a few directions in which further study can be done.
\begin{enumerate}
    \item We borrow this table from {\cite[Table $1$]{buck2023flattened}}, which gives the number of flattened $2$-Stirling permutations based on the number of runs.
\begin{table}[ht]
\centering
\resizebox{\textwidth}{!}{
\begin{tabular}{c | c | c | ccccccccc |}
    $n$ & $|\stir_n|$ & $|\flt(\stir_n)|$ &$|\flt_{1}(\stir_n)|$ & $|\flt_{2}(\stir_n)|$ & $|\flt_{3}(\stir_n)|$ & $|\flt_{4}(\stir_n)|$ & $|\flt_{5}(\stir_n)|$ & $|\flt_{6}(\stir_n)|$ & $|\flt_{7}(\stir_n)|$ &
    \\
    \hline
    1 & 1& 1 & 1 &  &  & & \\
    2 & 3 & 2 & 1 & 1 & & & \\
    3 &15 &6 & 1 & 5 &  &  \\
    4 &105 & 24 &1 & 15 & 8 & &  \\
    5 &945 & 116 &1 & 37 & 70 & 8 & \\
    6 &10395 & 648 &1 & 83 & 374 & 190 &  \\
    7 &135135& 4088 & 1 & 177 & 1596 & 2034 & 280\\
    8 &2027025& 28640 & 1 & 367 & 6012 & 15260 & 6720 & 280 \\
    9 & 34459425& 219920 & 1 & 749 & 20994 & 93764 & 88732 & 15680\\
    10 &654729075 & 1832224 & 1 & 1515 & 69842 & 508538 & 866796 & 363132 & 22400\\
\end{tabular}
}
\caption{Counts for flattened Stirling permutations based on number of runs.}\label{tab:data}
\end{table}\\
A sequence $(u_j)_{j=1}^{n}$ is said to be unimodal if it is an index $t$ such that $u_1\le u_2\le \cdots\le u_t$ and $u_t\ge u_{t+1}\ge\cdots \ge u_n$.  The Eulerian numbers $A_{n,k}$, the number of permutations in $\SSS_n$ with $k$ descents, is known to be unimodal (See \cite{petersen-eulerian-nos-book}). Notice that the run numbers (the rows of the Table \ref{tab:data}) for $1\le n\le 10$ are unimodal. A polynomial is said to be unimodal if its coefficients form a unimodal sequence. Based on this data, one can ask if the $\flt$-Eulerian polynomials of order $2$ are unimodal for all $n\in \mathbb N$.
\item For $n=3,4,5$ and $k=3$, we have $$A^{\flt}_{3,3}(t)=1+9t+2t^2, A^{\flt}_{4,3}(t)=1+26t+36t^2, A^{\flt}_{5,3}(t)=1+63t+251t^2+90t^3.$$
For $n=3,4,5$ and $k=4$, we have $$A^{\flt}_{3,4}(t)=1+13t+6t^2, A^{\flt}_{4,4}(t)=1+37t+84t^2+6t^3, A^{\flt}_{5,4}(t)=1+89t+546t^2+372t^3.$$

The polynomials are unimodal as well. Based on this data, one can ask if the $\flt$-Eulerian polynomials of order $k$ are unimodal for all $k$.
\item One of the interesting properties of the Eulerian polynomial $A_n(t)$ is that it is real rooted (See \cite{bona-cop-book}). One can ask the same question for the $\flt$-Eulerian polynomials as well. As it turns out, for $1\le n\le 10$, the $\flt$-Eulerian polynomials of order $2$ are all real rooted. One can ask if the order $2$ $\flt$-Eulerian polynomials are all real rooted. This question can also be extended to arbitrary $k$. For $n=3,4,5$ and $k=3,4$, the $\flt$-Eulerian polynomials are all real rooted. 
\item Other statistics over the $r$-multipermutations have been studied. Two such are the $\mathrm{plateau}$ and $\mathrm{ascent}$ statistics. For a flattened $k$-Stirling permutation $W=a_1\dots a_{nk}$ of order $n$, the number of indices $i\in [nk-1]$ such that $a_i=a_{i+1}$ (resp., $a_i<a_{i+1}$) is $\mathrm{plat}(W)$ (resp., $\mathrm{asc}(W)$). It would be interesting to calculate the joint distributions of these statistics. 
\end{enumerate}

\bibliographystyle{acm}

\begin{thebibliography}{10}

\bibitem{Andrealtrun}
{\sc Andr{\'e}, D.}
\newblock {\'E}tude sur les maxima, minima et s{\'e}quences des permutations.
\newblock {\em Ann. Sci. {\'E}c. Norm. Sup{\'e}r. 3(1)\/} (1884), 121--135.

\bibitem{species-bergeron}
{\sc Bergeron, F., Labelle, G., and Leroux, P.}
\newblock {\em Combinatorial species and tree-like structures. {Transl}. from
  the {French} by {Margaret} {Readdy}}, vol.~67 of {\em Encycl. Math. Appl.}
\newblock Cambridge: Cambridge University Press, 1998.

\bibitem{bona-cop-book}
{\sc B{\'o}na, M.}
\newblock {\em Combinatorics of Permutations}, 3rd~ed.
\newblock Boca Raton, FL: CRC Press, 2022.

\bibitem{bonaaltrun}
{\sc B{\'o}na, M.}
\newblock Generating functions of permutations with respect to their
  alternating runs.
\newblock {\em S{\'e}minaire Lotharingien de Combinatoire, B85b} (2021), 5 pp.

\bibitem{branden-actions_on_perms_unimodality_descents}
{\sc Br{\"a}nd{\'e}n, P, M.}
\newblock Actions on permutations and unimodality of descent polynomials.
\newblock {\em European Journal of Combinatorics 29 (2)\/} (2008), 514--534.

\bibitem{buck2023flattened}
{\sc Buck, A., Elder, J., Figueroa, A.~A., Harris, P.~E., Harry, K., and
  Simpson, A.}
\newblock Flattened stirling permutations.
\newblock{\em available at \url{https://arxiv.org/pdf/2306.13034.pdf}} (2023), 15 pages.

\bibitem{callan-flatten}
{\sc Callan, D.}
\newblock Pattern avoidance in ``flattened'' partitions.
\newblock {\em Discrete Math. 309}, 12 (2009), 4187--4191.

\bibitem{canfieldwilfalternatingrun}
{\sc Canfield, E.~R., and Wilf, H.}
\newblock Counting permutations by their alternating runs.
\newblock {\em Journal of Combinatorial Theory, Series A 115\/} (2008),
  213--225.

\bibitem{foata-desarm_loday}
{\sc Foata, D., and D{\'e}sarm{\'e}nien, J.}
\newblock The signed {E}ulerian numbers.
\newblock {\em Discrete Mathematics 99\/} (1992), 49--58.

\bibitem{foata-schutzenberger-eulerian}
{\sc Foata, D., and Sch{\"u}tzenberger, M.-P.}
\newblock {\em Th{\'e}orie g{\'e}om{\'e}trique des polyn{\^o}mes
  {E}ul{\'e}riens}, available at
  \url{https://www.mat.univie.ac.at/~slc/books/foaschuetz1.html}.
\newblock Lecture Notes in Mathematics, 138, Berlin, Springer-Verlag, 1970.

\bibitem{foata-strehl-actiontansec}
{\sc Foata, D., and Strehl, V.}
\newblock Rearrangements of the symmetric group and enumerative properties of
  the tangent and secant numbers.
\newblock {\em Math. Z. 137\/} (1974), 257--264.

\bibitem{foata-strehl}
{\sc Foata, D., and Strehl, V.}
\newblock Euler numbers and variations of permutations.
\newblock {\em Colloquio Internazionale sulle Teorie Combinatoire (Roma 1973)
  Tomo I Atti dei Convegni Lincei, No 17, Accad. Naz. Lincei, Rome\/} (1976),
  119--131.

\bibitem{gessel-stan-stirling}
{\sc Gessel, I., and Stanley, R.~P.}
\newblock Stirling polynomials.
\newblock {\em J. Comb. Theory, Ser. A 24\/} (1978), 24--33.

\bibitem{maaltrun}
{\sc Ma, S.-M.}
\newblock An explicit formula for the number of permutations with a given
  number of alternating runs.
\newblock {\em Journal of Combinatorial Theory, Series A 119\/} (2012),
  1660--1664.

\bibitem{Macmajor}
{\sc MacMahon, P.}
\newblock The indices of permutations and the derivation therefrom of functions
  of a single variable associated with the permutations of any assemblage of
  objects.
\newblock {\em American Journal of Mathematics 35\/} (1913), 314--321.

\bibitem{mansour-flatten}
{\sc Mansour, T., Shattuck, M., and Wang, D. G.~L.}
\newblock Counting subwords in flattened permutations.
\newblock {\em J. Comb. 4}, 3 (2013), 327--356.

\bibitem{nabawanda-flatten}
{\sc Nabawanda, O., Rakotondrajao, F., and Bamunoba, A.~S.}
\newblock Run distribution over flattened partitions.
\newblock {\em J. Integer Seq. 23}, 9 (2020), article 20.9.6, 14.

\bibitem{petersen-eulerian-nos-book}
{\sc Petersen, T.~K.}
\newblock {\em Eulerian Numbers}, 1st~ed.
\newblock Birkh{\" a}user, 2015.

\bibitem{sagan-art-count}
{\sc Sagan, B.~E.}
\newblock {\em Combinatorics: the art of counting}, vol.~210 of {\em Grad.
  Stud. Math.}
\newblock Providence, RI: American Mathematical Society (AMS), 2020.

\bibitem{wang-colored-set}
{\sc Wang, D. G.~L.}
\newblock On colored set partitions of type {{\(B_n\)}}.
\newblock {\em Cent. Eur. J. Math. 12}, 9 (2014), 1372--1381.
\end{thebibliography}
(Concerned with Sequences \seqnum{A007405}, \seqnum{A355164}, \seqnum{A355167})
\end{document}